\documentclass[english,openary,12pt]{article}
\usepackage{cmap}
\usepackage[english]{babel}
\usepackage{graphicx}
\DeclareGraphicsExtensions{.png}
\usepackage{floatflt}
\usepackage{amsfonts}                              
\usepackage{amsthm}                          
\usepackage{amsmath}
\usepackage{mathtools}                 
\usepackage{gensymb}
\usepackage{amssymb,amsfonts}
\usepackage[usenames]{color}
\usepackage{wrapfig}
\usepackage{url}
\usepackage{cmap}
\usepackage{graphicx}
\usepackage{multirow}
\usepackage{comment}
\usepackage{caption}
\usepackage{subcaption}
\usepackage{float}
\usepackage{diagbox}
\usepackage{enumerate}
\usepackage{titling}
\usepackage{makecell}
\usepackage{booktabs}
\usepackage{amsthm}
\usepackage{bm}
\usepackage{subcaption}
\renewcommand{\emptyset}{\varnothing}
\usepackage[hidelinks,hyperindex,breaklinks]{hyperref}
\usepackage{enumitem}
\usepackage[nottoc,numbib]{tocbibind} 
\setlist[itemize]{noitemsep, topsep=3pt, leftmargin=18pt}
\setlist[enumerate]{noitemsep, topsep=3pt, leftmargin=22pt}
\usepackage[title]{appendix}

\allowdisplaybreaks

\usepackage{tikz}
\usetikzlibrary{positioning, fit, arrows, arrows.meta,
calc, decorations.pathreplacing, backgrounds}
\usepackage{pgfplots}

\usepackage{soul}               
\setul{}{1.3pt}
\setstcolor{blue}

\makeatletter
\renewcommand\section{\@startsection {section}{1}{\z@}%
                                   {0ex \@plus 0ex \@minus 0ex}%
                                   {0.5ex \@plus0ex}%
                                   {\normalfont\LARGE\bfseries}}
\renewcommand\subsection{\@startsection{subsection}{2}{\z@}%
                                     {0ex\@plus 0ex \@minus 0ex}%
                                     {0.5ex \@plus 0ex}%
                                     {\normalfont\Large\bfseries}}
\renewcommand\subsubsection{\@startsection{subsubsection}{3}{\z@}%
                                     {0ex\@plus 0ex \@minus 0ex}%
                                     {0.5ex \@plus 0ex}%
                                     {\normalfont\large\bfseries}}
\makeatother

\makeatletter
\def\thm@space@setup{%
  \thm@preskip=2mm
  \thm@postskip=\thm@preskip 
}
\makeatother

\newcommand{\red}[1]{\textcolor{black}{#1}}

\newcommand{\brown}[1]{\textcolor{black}{#1}}

\usepackage[left=1.65cm,right=1.65cm,
    top=1.6cm,bottom=1.6cm,bindingoffset=0cm]{geometry}

\usepackage[titles]{tocloft}
\theoremstyle{definition}
\newtheorem{remark}{Remark}

\newtheorem{example}{Example}
\newtheorem{definition}{Definition}

\newtheorem{theorem}{Theorem}
\newtheorem{lemma}{Lemma}
\newtheorem{corollary}{Corollary}
\newtheorem{proposition}{Proposition}

\newtheorem*{remark-conjecture}{Remark-conjecture}

\newcommand{\R}{\mathbb{R}}
\newcommand{\N}{\mathbb{N}}

\begin{document}

\setlength{\abovedisplayskip}{5pt}
\setlength{\belowdisplayskip}{5pt}
\setlength\abovedisplayshortskip{5pt}
\setlength\belowdisplayshortskip{5pt}

\author{Ivan Novikov, Université Paris Dauphine (CEREMADE)}
\title{\vspace{-2.1cm} Zero-Sum State-Blind Stochastic Games \\with Vanishing
Stage Duration
\vspace{-0.42cm}}
\date{}

\maketitle

\vspace{-1.8cm}

\begin{abstract}
In stochastic games with stage duration $h$, players act at times $\red{0}, h, 2h$, \brown{and so on.} The payoff and leaving probabilities are proportional to $h$. \brown{As $h$ approaches $0$}, such \brown{discrete-time} games approximate games \brown{played} in continuous time. The behavior of the values when $h$ tends to $0$ was already studied in the case of stochastic games with perfect observation of the state. 

We \brown{examine} the same question for the case of state-blind stochastic games. Our main \brown{finding is that, as $h$ approaches $0$,} the value of any state-blind stochastic game with stage duration $h$ converges to \red{the} unique viscosity solution of a partial differential equation.

\vspace{1cm}

\end{abstract}

\vspace{1cm}

\tableofcontents

\newpage

\textbf{Notation:}
\vspace{-0.15cm}
\begin{itemize}
	\item $\N^*$ is the set of all positive integers;
	\item $\R_+ := \{x : x \in \R \text{ and } x \ge 0\}$;
	\item If $f(x)$ is a function defined on a set $X$, then 
	$\|f(x)\|_\infty := \sup_{x \in X}|f(x)|$;
	\item If $f(x)$ is a function defined on a finite set $X$, then 
	$\|f(x)\|_1 := \sum_{x \in X}|f(x)|$;
	\item If $C$ is a finite set, then 
	$\Delta (C)$ is the set of probability measures on $C$;
	\item If $X$ is a finite set, and $f,g : X \to \R$ are two functions, then
	$$\langle f (\cdot), g (\cdot)\rangle := \sum\nolimits_{x\in X} f(x) g(x);$$
  If $x = (x_1, \ldots, x_n), y = (y_1, \ldots, y_n) \in \R^n$, then 
  $\langle x, y \rangle := \sum_{i=1}^n x_i y_i$;
	\item If $X$ is a finite set, $\zeta \in \Delta(X)$, and 
	$\mu$ is a $|X| \times |X|$ matrix, then for any $x\in X$
	$$\left(\zeta * \mu\right)(x) := 
	\sum\nolimits_{x'\in X} \zeta(x') \cdot \mu_{x'x};$$
	\item If $I, J$ are \red{finite} sets and $g : I \times J \to \R$ is a function, then 
	\begin{multline*}\hspace{-1cm}
  \texttt{Val}_{I \times J} [g(i,j)] := 
	\sup_{x \in \Delta(I)} \inf_{y \in \Delta(J)} 
	\left(\int_{I \times J} g (i, j) \; dx (i) \otimes dy (j)\right) =
	\inf_{y \in \Delta(J)} \sup_{x \in \Delta(I)} 
	\left(\int_{I \times J} g (i, j) \; dx (i) \otimes dy (j)\right),
  \end{multline*}

	i.e. $\texttt{Val}_{I \times J} [g(i,j)]$ is the value of \brown{the} one-shot zero-sum game with action spaces $I, J$ and with payoff function $g$.
\end{itemize}

\section{Introduction}

\textbf{Zero-sum stochastic games} 
with perfect observation of the state were first defined in \cite{Sha53}. Such a game is played in discrete time as follows. At the beginning of each stage, player~1 and player~2 observe the current state and remember the actions taken in previous stages. They then choose their mixed actions, after which player~1 receives some payoff, depending on players' actions and the current state. Player~2 receives the opposite of this payoff. The next state is chosen according to some probability law, depending on players' actions and the current state. For a fixed discount factor $\lambda \in (0,1]$, player 1 aims to maximize the $\lambda$-discounted total payoff
$E\left(\lambda \sum\limits_{n = 1}^\infty (1-\lambda)^{n-1} g_n\right),$
where $g_n$ is $n$-th stage payoff; conversely, player 2 aims to minimize it. Under standard assumptions, maxmin and minmax coincide, and the resulting quantity is called the value, denoted by $v_\lambda$.

A similar model of stochastic games, in which players cannot observe the current state, is called state-blind stochastic games. In such games, the players can observe only the initial probability distribution on the states and the previous actions. One can define the value $v_\lambda$ in the same way as above.

An analogous model in continuous time is \textbf{continuous-time Markov games}, in which players are allowed to choose actions at any moment of time. Players' actions at time $t$ may depend only on the current state and on $t$ (with some technical measurability conditions). Such games were introduced in \cite{Zac64}, and later studied in many other papers, for example, in \cite{GuoHer03} and \cite{GuoHer05}. In a continuous-time Markov game, player 1 tries to maximize the $\lambda$-discounted total payoff
$E(\int_0^{+\infty} \lambda e^{-\lambda t} g_t dt),$
where $g_t$ is instantaneous payoff at time $t$. Player 2 tries to minimize it. Just as in discrete-time games, under standard assumptions maxmin and minmax coincide, and the resulting quantity is called the value. One can also consider a more general total payoff $E(\int_0^{+\infty} k(t) g_t dt),$ where $k$ is a nonincreasing continuous function.

The article \cite{Sor17} considers discretizations of continuous-time Markov games. In a discretization, players can act only at specific times $t_1, t_2, t_3, \ldots$, rather than at any moment of time, as in a standard continuous-time Markov game. However, the state variable continues to change as it would in the continuous-time model. The article \cite{Sor17} considers both the case of perfect observation of the state and the state-blind case. For each of these cases it is proved that if $\sup\{t_{i+1} - t_i\}$ tends to $0$, then the value of the discretization of a continuous-time Markov game converges to the unique viscosity solution of a differential equation.

\textbf{Zero-sum stochastic games with perfect observation of the state 
and with stage duration} were introduced in \cite{Ney13}. They propose another way to approximate continuous-time games. Starting from a stochastic game $\Gamma$ (with perfect observation of the state) with stage duration $1$, Neyman considers a family $\Gamma_h$ of stochastic games in which players act at time $0, h, 2h,$ and so on. In these games, both the payoffs and leaving probabilities are normalized at each stage, making them proportional to $h$. This gives the value $v_{h,\lambda}$ depending both on the stage duration $h$ and the discount factor $\lambda$. A particular case of interest is when $h$ is small, which approximates a game played in continuous time with $\lambda$-discounted payoff $\int_0^{+\infty} \lambda e^{-\lambda t} g_t dt$.

In \cite{Ney13}, Neyman considered the asymptotics of $v_{h,\lambda}$ when $h$ tends to $0$. Among other things, it was proved that when $h$ tends to $0$, the value of a finite game with stage duration tends to the unique solution of a functional equation. Subsequently, the article \cite{SorVig16} generalized some of the results from \cite{Ney13} and obtained new findings; it considers the case in which the state and action spaces may be compact, and stage durations $h_n$ may depend on the stage number $n$. In our paper, we also assume that stage duration may depend on the stage number.

In this framework of games with stage duration \cite{Ney13}, we introduce more general total payoffs for a game with stage duration. When $\sup h_i$ is small, these payoffs approximate the continuous-time game with the total payoff $\int_0^{+\infty} k(t) g_t dt,$ where $k$ is a nonincreasing continuous function. Our \textbf{Theorem~\ref{xmn444}} states that the value of such a stochastic game with stage duration converges uniformly, as $\sup h_i$ tends to $0$, to the unique viscosity solution of a differential equation. This theorem generalizes an already known result for discounted games (\cite[Theorem~1]{Ney13}, \cite[Corollary 7.1]{SorVig16}). The proof is based on a result from \cite{Sor17}.

The study of \textbf{zero-sum state-blind stochastic games with stage duration} is the main goal of this article. In \S\ref{def1}, we give a natural definition of such games. Our main result, \textbf{Theorem~\ref{thh1}}, states that when $\sup h_i$ tends to $0,$ the value of a state-blind stochastic game with stage duration converges uniformly to the unique viscosity solution of a partial differential equation. Thus Theorem~\ref{thh1} is an analogue of Theorem~\ref{xmn444} for state-blind stochastic games. In the particular case of discounted games the equation is autonomous (does not depend on $t$), see Corollary~\ref{thh3}. The proof of Theorem~\ref{thh1} has the same structure as a similar result in \cite{Sor17}.

\section{Organization of the paper}

In \S\ref{stoch-intro}, we provide all the necessary information about stochastic games with perfect observation of the state. In \S\ref{one-h}, we define stochastic games with perfect observation of the state and with stage duration, which generalize the discounted case considered in literature before. In \S\ref{ther5}, we state Theorem~\ref{xmn444}, \red{whose} proof is given in \S\ref{xmn555}.

In \S\ref{repeatedGames}, we recall and discuss the model of state-blind stochastic games. In \S\ref{def1}, we introduce state-blind stochastic games with stage duration. In \S\ref{jjjnt44}, we state Theorem~\ref{thh1}, \red{whose} proof is given in \S\ref{thh10}.

In \S\ref{disc}, we give some results about the discounted case, in both the framework of perfect observation of the state and the state-blind framework. In \S\ref{the_final}, we give some final comments.

\newpage

\section{Games with stage duration (perfect observation of the state)}
\label{second}

\subsection{Zero-sum stochastic games with perfect observation of the state}
\label{stoch-intro}

In this section, we introduce all the necessary notions from the theory of  zero-sum stochastic games with perfect observation of the state. This section is partially based on books \cite{MerSorZam15}, \cite{LarRenSor19}, and \cite{Sor02}. 

Now, we give a minor generalization of a construction from \cite{Sha53}. A \emph{zero-sum stochastic game with perfect observation of the state} is a $5$-tuple 
$(\Omega, I, J, \{g_m\}_{m\in\N^*}, \{P_m\}_{m\in\N^*})$, where: 
\begin{itemize}
	\item $\Omega$ is a finite non-empty set of states;
	\item $I$ is a finite non-empty set of actions of player~1; 
	\item $J$ is a finite non-empty set of actions of player~2;
	\item $g_n : I \times J \times \Omega \to \R$ is the $n$-th stage payoff 
	function of player~1;
	\item $P_n : I \times J \to 
\{\text{row-stochastic matrices } |\Omega| \times |\Omega|$\} is a 
transition probability function at the $n$-th stage.
\end{itemize}

Recall that a matrix $A = (a_{i j})$ is called row-stochastic if $a_{i j} \ge 0$ for all $i, j$, and $\sum_{j} a_{i j} = 1$ for any fixed $i$.
Note that in our model we allow $g_n$ and $P_n$ to be dependent on the stage number $n$, which is not the case in the original model from \cite{Sha53}. 

We denote by $P_n (i,j) (\omega_a, \omega_b)$ the $(\omega_a,\omega_b)$-th element of the matrix $P_n (i,j)$.


A stochastic game $(\Omega, I, J, \{g_m\}_{m\in\N^*}, \{P_m\}_{m\in\N^*})$ proceeds in stages as follows. The initial state $\omega_1$ is known to the players. At each stage $n \in \N^*$:
\begin{enumerate}
	\item Players observe the current state $\omega_n$ and remember actions of each other at the previous stage.
	\item Players simultaneously choose mixed actions. Player~1 chooses $x_n \in \Delta(I)$ and player~2 chooses $y_n \in \Delta(J)$.
	\item An action $i_n \in I$ of player~1 (respectively $j_n \in J$ of player~2) is chosen according to the probability measure $x_n \in \Delta(I)$ (respectively $y_n \in \Delta(J)$). They are observed by both players.
	\item Player~1 obtains a payoff $g_n = g_n(i_n, j_n, \omega_n)$, while player~2 obtains payoff $-g_n$. The new state $\omega_{n+1}$ is chosen according to the probability law $P_n = P_n(i_n, j_n)(\omega_n, \cdot)$.
\end{enumerate}

The above description of the game is assumed to be common knowledge.

\emph{A history of length $n \in \N$} for the stochastic game 
$(\Omega, I, J, \{g_m\}_{m\in\N^*}, \{P_m\}_{m\in\N^*})$ is\\
$(\omega_1, i_1, j_1, \omega_2, i_2, j_2, \ldots, \omega_{n-1}, i_{n-1}, j_{n-1}, \omega_n)$.
The set of all histories of length $n$ is 
$H_n := \Omega \times (I \times J \times \Omega)^{n-1}$.
A \emph{(behavior) strategy} of player~1 (respectively player~2) is a function 
$\sigma : \bigcup_{n \ge 1} H_n \to \Delta(I)$ 
(respectively $\tau : \bigcup_{n \ge 1} H_n \to \Delta(J)$). 
Players' strategies induce probability distribution on the set 
$\Omega \times (I \times J \times \Omega)^{\N^*}$. 
(Indeed, strategies induce a probability distribution on the set $H_1$, 
then on the set $H_2$, etc. By Kolmogorov extension theorem, 
this probability can be extended in a unique way to the set 
$\Omega \times (I \times J \times \Omega)^{\N^*}$). 
In particular, given an initial state $\omega \in \Omega$, 
strategies $\sigma : \bigcup_{n \ge 1} H_n \to \Delta(I), 
\tau : \bigcup_{n \ge 1} H_n \to \Delta(J)$, and the induced probability 
distribution $P^\omega_{\sigma, \tau}$ on $\Omega \times 
(I \times J \times \Omega)^{\N^*}$, we can consider the expectation 
$E^\omega_{\sigma, \tau}$ of a random variable on $\bigcup_{n \ge 1} H_n$.

Now, we need to choose how to compute a total payoff function. Given a stochastic game
$\Gamma = (\Omega, I, J, \{g_m\}_{m\in\N^*}, \{P_m\}_{m\in\N^*})$ 
and a sequence $\{b_m\}_{m \in \N^*}$ with $b_m \ge 0 \; (m \in \N^*)$ and with $0 < \sum_{n=1}^\infty b_n < \infty$, we consider the stochastic game $\Gamma(\{b_m\}_{m \in \N^*})$ 
with total payoff
$E^\omega_{\sigma, \tau} \left(\sum_{n=1}^\infty b_n g_n\right).$
(It depends on a strategy profile $(\sigma, \tau)$ and an initial state $\omega$).

Some particular cases of the above definitions have its own name. For a fixed 
$T \in \N^*$, the \emph{(normalized) repeated $T$ times game $\Gamma^T$} is obtained 
if we set in the above definition
$b_n = 1/T$ for $n = 1, 2, \ldots, T$, and $b_n = 0$ for $n \ge T+1$.
For a fixed $\lambda \in (0,1)$, the 
\emph{(normalized) $\lambda$-discounted\footnote{Note that in the given definition the weight of $n$-th stage is $(1-\lambda)^{n-1}$. In the economic literature, the weight of $n$-th stage is often $\lambda^{n-1}$.} 
game 
$\Gamma^\lambda$} is obtained if we set in the above definition
$b_n = \lambda (1-\lambda)^{n-1}$ for all $n \in \N^*$.

Analogously to single-shot zero-sum games, we may define the value and 
the $(\varepsilon$-$)$optimal strategies of stochastic games.
Given a stochastic game 
$\Gamma = (\Omega, I, J, \{g_m\}_{m\in\N^*}, \{P_m\}_{m\in\N^*})$,
the stochastic game $\Gamma(\{b_m\}_{m \in \N^*})$ with payoff 
$\sum_{n=1}^\infty b_n g_n$ is said to have a value 
$V : \Omega \to \R$ if for all $\omega \in \Omega$ we have
$V(\omega) = \sup\limits_{\sigma} \inf\limits_{\tau} E^\omega_{\sigma, \tau} 
\left(\sum\nolimits_{n=1}^\infty b_n g_n\right) = 
\inf\limits_{\tau} \sup\limits_{\sigma} E^\omega_{\sigma, \tau} 
\left(\sum\nolimits_{n=1}^\infty b_n g_n\right).$



Sometimes instead of the transition probability 
function $P$, we consider the \emph{kernel}\\
$q : I \times J \to 
\{\text{matrices } |\Omega| \times |\Omega|$\}
defined
by the expression
$$q(i,j)(\omega, \omega') = 
\begin{cases}
P(i,j)(\omega, \omega') &,\text{if } \omega \neq \omega';\\
P(i,j)(\omega, \omega) - 1 &,\text{if } \omega = \omega'.
\end{cases}$$
In particular, we are sometimes going to define a stochastic game by using 
kernels instead of transition probability functions. 
Note that for any $\omega \in \Omega$ we have
$\sum_{\omega' \in \Omega} q(i,j)(\omega, \omega') = 0$.

\subsection{Zero-sum stochastic games with perfect observation of the state 
and with stage duration}
\label{one-h}

In this subsection, we introduce games with stage duration. 

Let us fix the notation that will be used during the entire section. Let $T$ be either a positive number or $+\infty$. Let $T_\infty$ be a partition of $[0,T)$; in other words, $T_\infty$ is a strictly increasing sequence $\{t_n\}_{n\in\N^*}$ such that $t_1 = 0$ and $t_n \xrightarrow{n \to \infty} T$. For each given partition $T_\infty$, denote $h_n = t_{n+1}-t_n$ for any $n \in \N^*$.

Throughout this entire subsection, we fix a stochastic game $(\Omega, I, J, g, q)$, where $q$ is the kernel of the game. (We assume that the payoff function $g$ and the kernel $q$ do not depend on the stage number).

The state space $\Omega$ and the action spaces $I, J$ of both players are independent of the partition $T_\infty$ in a stochastic game with $n$-th stage duration $h_n$. The payoff function and the kernel depend on $T_\infty$. For $n \in \N^*$, the $n$-th stage payoff function is $g_n = h_n g$, and the $n$-th stage kernel function is $q_n = h_n q$. The following definition summarizes this.
\begin{definition}
Given a stochastic game $(\Omega, I, J, g, q)$ and a partition $T_\infty$ of $[0,T)$, the \emph{stochastic game with perfect observation of the state and with $n$-th stage duration $h_n$} is the stochastic game
$(\Omega, I, J, \{h_m g\}_{m\in\N^*}, \{h_m q\}_{m\in\N^*}).$
\end{definition}

Now we define a total payoff. Fix a nonincreasing continuous function $k : [0,T] \to \R_+$ with $\int_{0}^T k(t) dt = 1$.

\begin{definition}
\label{ty3ee}
\red{Given a stochastic game $(\Omega, I, J, g, q)$ and a partition $T_\infty$ of $[0,T)$,} \emph{the stochastic game 
$(\Omega, I, J, g, q)$
with perfect observation of the state and 
with $n$-th stage duration $h_n$, weight function 
$k(t)$, initial time $t_n \in T_\infty$, and initial state $\omega$}, 
is the stochastic game 
$(\Omega, I, J, \{h_m g\}_{m\in\N^*}, \{h_m q\}_{m\in\N^*})$
with total payoff
$G_{T_\infty,k}(t_n,\omega) := E_{\sigma,\tau}^{\omega}
\left( \sum\limits_{i = n}^\infty h_i k(t_i) g_i \right).$
Denote by $v_{T_\infty,k}(t_n,\omega)$ the value of the game with such a 
total payoff.

If $t \in [t_n,t_{n+1}]$ and 
$t = \alpha t_n + (1-\alpha) t_{n+1}$, we then define
$$v_{T_\infty,k}(t,\omega) = 
\alpha v_{T_\infty,k}(t_n,\omega) + 
(1-\alpha) v_{T_\infty,k}(t_{n+1},\omega).$$
\end{definition}

We are interested in the behavior of such games when the duration of each stage is vanishing, i.e. we want to know what happens when $\sup h_i \to 0$.

\begin{remark}[Why such a definition]
	If all $h_i$ are small, then the value of a stochastic game with $n$-th stage duration $h_n$, weight function $k(t)$, and initial time $s \in \R_+$, is close to the value of	the continuous-time game with total payoff $\int_s^\infty k(t) g_t dt$. See Propositions \ref{h3er} and \ref{4cd6YYu} in \S\ref{the_final}.
\end{remark}

\begin{example}[Discounted games]
If initial time is $0$, $T=+\infty$, and $k(t) = \lambda \exp\{-\lambda t\}$,  we obtain the normalized $\lambda$-discounted stochastic game with stage duration. We discuss it in more details in \S\ref{disc}.
\end{example}

\begin{example}[Repeated finitely number of times games]
If we take $T \in \N^*$, initial time $0$, $k(t) = 1/T$ for $t\in [0,T]$, $h_n = 1$ for $n = 1, \ldots, T$, and $h_n = 0$ for $n > T$, then we recover the case of normalized repeated $T$ times games with total payoff $\frac{1}{T}\sum_{i = 1}^T g_i$. \red{Note that formally such a choice of $h_n$ does not satisfy the definition, because $h_n$ are assumed to be strictly positive. To fix it, we may consider a small $\varepsilon$ and define $h_n = 1 - \varepsilon$ for $n = 1, \ldots, T$, and $h_n = (1-\varepsilon) \varepsilon^{n-T}$ for $n > T$. The closer $\varepsilon$ is to $0$, the closer we are to the case of normalized repeated $T$ times games.}
\end{example}

\begin{remark}
\label{important}
Note that
$\sum\nolimits_{n = 1}^\infty h_n k(t_n)$
is a (left) Riemann sum for the integral
$\int_{0}^T k(t) dt.$
Hence for any $\varepsilon > 0$ there is $\delta > 0$ such that for any
partition $T_\infty$ with $\sup_{i\in\N^*} h_i < \delta$, we have
$$\left\|\sum_{n = 1}^\infty h_n k(t_n) g_n\right\|_\infty \le
\left|\sum_{n = 1}^\infty h_n k(t_n)\right| 
\left\|g\right\|_\infty \le
\left(\int_{0}^T k(t) dt + \varepsilon\right) \left\|g\right\|_\infty = 
(1+\varepsilon) \left\|g\right\|_\infty.$$
\end{remark}

\subsection{Theorem~\ref{xmn444}: In the case of perfect observation of the state, the uniform limit 
$\lim_{\sup_{i\in\N^*} h_i \to 0} v_{T_\infty,k}$ 
is the unique viscosity solution of some differential equation}
\label{ther5}

First, we recall the definition of a viscosity solution.

\begin{definition}
Given a function $H : [0,T) \times \Omega \times \R^{n+1} \to \R$, a function $u : [0,T) \times \Omega \to \R$ is called a \emph{viscosity solution} of the differential equation 
$$0 = \frac{d}{d t} u (t,\omega) + H(t, \omega, u(t,\omega), \nabla u(t,\omega))$$
if:
\begin{enumerate}
\item for any $\mathcal C^1$ function $\psi : [0,T) \times \Omega \to \R$ with  $u - \psi$ having a strict local maximum at $(\overline{t}, \overline{\omega}) \in [0,T) \times \Omega$ we have
$0 \le \frac{d}{d t} \psi (\overline{t}, \overline{\omega}) + H(\overline{t}, 
\overline{\omega}, \psi(\overline{t}, \overline{\omega}), \nabla \psi(\overline{t}, \overline{\omega}))$;
\item for any $\mathcal C^1$ function $\psi : [0,T) \times \Omega \to \R$ with 
$u - \psi$ having a strict local minimum at $(\overline{t}, \overline{\omega}) \in 
[0,T) \times \Omega$ we have
$0 \ge \frac{d}{d t} \psi (\overline{t}, \overline{\omega}) + H(\overline{t}, 
\overline{\omega}, \psi(\overline{t}, \overline{\omega}), \nabla \psi(\overline{t}, \overline{\omega}))$.
\end{enumerate}
\end{definition}

\begin{theorem}
\label{xmn444}
If $(\Omega, I, J, g, q)$ is a finite stochastic game with perfect observation of the state, then the uniform limit
$\lim\limits_{\underset{h_1+h_2+\ldots = T}{\sup_{i\in\N^*} h_i \to 0}} 
v_{T_\infty,k}$ exists and is the unique viscosity solution of the differential equation (in $v(t,\omega)$)
\begin{equation}
\label{reenn3}
0 = \frac{d}{dt} v(t,\omega) + \texttt{Val}_{I\times J} [k(t) g(i,j,\omega) + 
\langle q(i,j)(\omega,\cdot)\,, v(t,\cdot) \rangle].
\end{equation}
\end{theorem}

\begin{remark}
	For the discounted case $k(t)=\lambda e^{-\lambda t}$, \eqref{reenn3} transforms into
	\begin{equation}
	\label{wwr5}
	\lambda v(\omega) = \texttt{Val}_{I\times J}\left[\lambda g(i,j,\omega) + 
	\langle q(i,j)(\omega,\cdot)\,, v(\cdot)\rangle\right].
	\end{equation} 
  (make a substitution $v(t,\omega) \mapsto e^{-\lambda t} v(\omega)$). For this particular case, Theorem~\ref{xmn444} is already known. The result \cite[Theorem~1]{Ney13} proves it for the uniform case in which $h_n = h$ for all $n$, while the result \cite[Corollary~7.1]{SorVig16} proves it for the general case.
\end{remark}

\begin{remark}
	The above theorem is similar to \cite[Proposition~4.3]{Sor17}, which concerns a similar result in another model of games with stage duration. \red{In addition, the differential equation \eqref{wwr5} also characterizes the value of the analogous game played in continuous time, see \cite[Theorem~5.1]{GuoHer05}.}
\end{remark}

\subsection{The proof of Theorem~\ref{xmn444}}
\label{xmn555}

\subsubsection{Preliminaries: Shapley operator}
\label{SHAP3}
The Shapley operator is a useful tool which uses the recursive structure of a stochastic game to express its value.

Fix an initial time $n \in \N^*$. We consider a finite stochastic game
$G_n = (\Omega, I, J, \{g_m\}_{m\in\N^*}, \{P_m\}_{m\in\N^*})$ with perfect observation of the state, where $\{P_m\}_{m\in\N^*}$ are transitional probability functions and $\sum_{i=1}^\infty \|g_i\|_\infty < \infty$, and with total payoff $E^\omega_{\sigma, \tau} \left(\sum\nolimits_{i=n}^\infty g_i\right).$ Denote by $v_n(\omega)$ the value of such a game.

Denote by $C(\Omega, \R)$ the set of continuous functions from $\Omega$ to $\R$.
For a sequence of maps $S_1, S_2, \ldots$ from a Banach space $C$ to 
itself, for any $z \in C$, we denote $\prod\nolimits_{i = 1}^\infty S_i (z) 
:= \lim\nolimits_{i \to \infty} (S_1 \circ S_2 \circ \cdots \circ S_i (z))$. 

\begin{proposition}[The Shapley operator and its properties]
  \label{shapley}
  Consider the described above game $G_n$. 
  Denote by $\psi_n \: (n \in \N^*)$ the operator
  \begin{align*}
    \psi_n : \; C(\Omega, \R) \to C(\Omega, \R), \quad
    f(\omega) \mapsto
    \texttt{Val}_{I \times J} 
    [g_n(i,j,\omega) + \langle P_n(i,j)(\omega,\cdot)\,, f(\cdot)\rangle].
  \end{align*}
  Then:
  \begin{enumerate}
    \item For each $n\in \N^*$ the operator $\psi_n$ is nonexpansive, i.e.
    $\|\psi_n(f - g)\|_\infty \le \|f-g\|_\infty$;
    \item If $\overline G_n = (\Omega, I, J, \{\overline g_m\}_{m\in\N^*}, \{\overline P_m\}_{m\in\N^*})$ is a finite stochastic game which has the same properties as the game $G_n$, and $\overline \psi_n$ is the operator
    \begin{align*}
      \overline \psi_n : \; C(\Omega, \R) \to
      C(\Omega, \R), \quad
      f(\omega) \mapsto
      \texttt{Val}_{I \times J} 
      \left[\overline g_n(i,j,\omega) + 
      \langle \overline P_n(i,j)(\omega,\cdot)\,, f(\cdot)\rangle\right],
    \end{align*}
    then\footnote{\red{We consider $P_n$ and $\overline P_n$ as functions of 4 variables $(i, j, \omega, \omega')$.}}
    $\left\|\psi_n(f) - \overline \psi_n(f)\right\|_\infty \le
    \left\| g_n - \overline g_n\right\|_\infty
    + \left\|P_n - \overline P_n\right\|_1 \cdot \|f\|_\infty;$
    \item The value $v_n(\omega)$ of the game $G_n$ exists and is equal to
    $\prod\nolimits_{i=n}^\infty \psi_i (0)$. (In particular, 
    such a product is well-defined);
    \item For any $n\in\N^*$ and any $\omega \in \Omega$ we have
    $v_n(\omega) = (\psi_n v_{n+1})(\omega);$
    \item For any $n \in \N^*$, there exists an optimal Markov strategy
    in the game $G_n$.
  \end{enumerate}  
\end{proposition}

\red{Recall that a strategy is said to be \emph{Markov} if the players' mixed actions at each stage depend only on the current stage number and on the current state.} An analogue of \red{the above} proposition for $\lambda$-discounted games and
finitely repeated games is well-known, see, for example,
\cite[Theorem~IV.3.2, Proposition~IV.3.3]{MerSorZam15}. The more general case
given here (for a more general total payoff $\sum\nolimits_{i=n}^\infty g_i$) can be proved analogously.

\begin{remark}
\label{rrff3}
(cf. \cite{MerSorZam15}).
We considered the case of a finite game $G$, but analogous proposition holds also in a more general setting, for example, if:
\begin{enumerate}
  \item $I$ and $J$ are finite;
  \item $\Omega$ is a compact metric space;
  \item $g$ is bounded;
  \item For each fixed $n \in \N^*$ and $i \in I$, 
  the function 
  $(j, \omega) \mapsto 
  \int_\Omega f(\omega) P_n(i,j)(\omega, \omega') d \omega'$ 
  is continuous for any 
  bounded continuous function $f$ on $\Omega$, and for each fixed 
  $n \in \N^*$ and $j \in J$, the function 
  $(i, \omega) \mapsto 
  \int_\Omega f(\omega) P_n(i,j)(\omega, \omega') d \omega'$ 
  is continuous for any bounded continuous function $f$ on $\Omega$.
\end{enumerate}
\end{remark}

\subsubsection{Preliminaries: The discretization of a zero-sum continuous-time Markov game}

\label{cc01}
Now, we briefly recall the model from \cite{Sor17}.
As before, $T$ is either a positive number or $+\infty$, and $T_\infty$ is a partition of $[0,T)$;  $h_n = t_{n+1}-t_n$ for each $n \in \N^*$.

We consider a continuous-time Markov game $(\Omega, I, J, g, q)$, in which players are only allowed to act at times $t_1, t_2, t_3, \ldots$ At any time $t \in [t_n, t_{n+1})$ players act according to their decision at time $t_n$. In the above game $\Omega$ is a finite state space, $I$ and $J$ are finite sets of actions of players, $g : I \times J \times \Omega \to \R$ is an instantaneous payoff function, and 
$q : I \times J \to \{|\Omega| \times |\Omega| \text{ matrices } A \text{ such that }a_{i j} \ge 0 \, \forall i \neq j, a_{i i} \le 0 \,\forall i, \text{ and } \sum_{j = 1}^{|\Omega|} a_{i j} = 0\}$ is an infinitesimal generator.

We fix a nonincreasing continuous function $k : [0,T] \to \R_+$ with $\int_{0}^T k(t) dt = 1$. The \emph{discretization of a continuous-time Markov game $(\Omega, I, J, g, q)$, in which player acts at times $t_1, t_2, \ldots$} is the stochastic game
$(\Omega, I, J, \{g_m\}_{m\in\N^*}, \{P_m\}_{m\in\N^*})$ with perfect observation of the state, where $P_m$ is the $m$-th stage transition probability function, and 
$$g_m(i,j,\omega) = \int_{t_m}^{t_{m+1}} k(t) g(i, j, \omega) dt
\quad \text{ and } \quad 
P_m(i,j)(\omega,\omega') = (\exp\{h_m q(i,j)\})(\omega,\omega').$$

Given two strategies $\sigma, \tau$ of players, an initial state $\omega$, an initial time $t_n \in T_\infty$,
the \emph{total payoff} is 
$$G^{\texttt{cont}}_{T_\infty,k}(t_n,\omega) := E_{\sigma,\tau}^{\omega}
\left(\int_{t_n}^{T} k(t) g(i_t,j_t, \omega_t) dt\right),$$ 
\red{where $\omega_t = \omega_{t_m}, i_t = i_{t_m}, j_t = j_{t_m}$ if $t \in [t_m, t_{m+1})$. The transition from $\omega_{t_m}$ to $\omega_{t_{m+1}}$ occurs according to the probability law $P_{m}$ described above.}
The game is said to have a 
\emph{value} $v^{\texttt{cont}}_{T_\infty,k}(t_n,\omega)$ if
$v^{\texttt{cont}}_{T_\infty,k}(t_n,\omega) = 
\sup_{\sigma} \inf_{\tau} G^{\texttt{cont}}_{T_\infty,k}(t_n,\omega) = 
\inf_{\tau} \sup_{\sigma} G^{\texttt{cont}}_{T_\infty,k}(t_n,\omega).$
We can define $G^{\texttt{cont}}_{T_\infty,k}(t,\omega)$ and 
$v^{\texttt{cont}}_{T_\infty,k}(t,\omega)$ for any $t\in [0,T)$ via linearity.
\red{Note that the value exists, because it is the value of a stochastic game with some specific payoff and transitional probability functions.}

The uniform limit
$\lim\limits_{\underset{h_1+h_2+\ldots = T}{\sup_{i\in\N^*} h_i \to 0}} 
v^{\texttt{cont}}_{T_\infty,k}$ is the unique viscosity solution of 
the differential equation
$0 = \frac{d}{dt} v(t,\omega) + \texttt{Val}_{I\times J} [k(t) g(i,j,\omega) + 
\langle q(i,j)(\omega,\cdot)\,, v(t,\cdot)\rangle],$ see \cite[Proposition~4.3]{Sor17}.

\subsubsection{The proof of Theorem~\ref{xmn444}}
The proof given below is a generalization of 
\cite[proofs of Lemma~8.1 and Proposition~8.1]{SorVig16}. 

\begin{proof}
Let $v^{\texttt{cont}}_{T_\infty,k}$ be the value of the discretization of
the continuous-time game
$(\Omega, I, J, g, q)$, where $q$ is an infinitesimal generator. 
Define for $n \in \N^*$
\begin{align*}
  &\psi_n^h : \; C(\Omega, \R) \to C(\Omega, \R),\\
  &f(\omega) \mapsto
  \texttt{Val}_{I \times J} 
  [k(t_n) h_n g(i,j,\omega) + 
  \langle(Id + h_n q(i,j))(\omega,\cdot)\,, f(\cdot)\rangle];\\
  &\overline \psi_n^h : \; C(\Omega, \R) \to C(\Omega, \R),\\
  &f(\omega) \mapsto
  \texttt{Val}_{I \times J} 
  \left[\int_{t_n}^{t_{n+1}} k(t) g(i, j, \omega) dt 
  + \langle \exp\{h_n q(i,j)\}(\omega,\cdot)\,, f(\cdot)\rangle\right].
\end{align*}

First, note that by Proposition~\ref{shapley}(2), for any continuous $f$ and any $n \in \N^*$ we have
\begin{equation}
  \label{eq401}
  \left\|\psi_n^h (f) - \overline \psi_n^h (f)\right\|_\infty \le 
  \left| k(t_n) h_n - \int_{t_n}^{t_{n+1}} k(t) dt\right| 
  \|g\|_\infty + 
  \left\|\left( Id + h_n q \right) - \exp\{h_n q\}\right\|_1 \cdot 
  \|f\|_\infty.
\end{equation}

By the mean value theorem for integrals, there exists $c\in (t_n, t_{n+1})$
such that for any $n \in \N^*$ we have
\begin{equation}
  \label{eq402}
  \frac{1}{h_n} \left| k(t_n) h_n - \int_{t_n}^{t_{n+1}} k(t) dt \right|
  = \left| k(t_n) - k(c) \right| \le
  k(t_n) - k(t_{n+1}).
\end{equation}

There exists $C>0$ such that for any $n \in \N^*$ we have
\begin{equation}
  \label{eq403}
  \frac{1}{h_n} \left\|\left(Id + h_n q\right) - \exp\{h_n q\}\right\|_1
  = \frac{1}{h_n} \left\|\left(Id + h_n q\right) - 
  \sum_{k=0}^\infty \frac{(h_n q)^k}{k !} \right\|_1 \le
  \frac{1}{h_n} C h_n^2 = C h_n.
\end{equation}
\red{(Recall that we consider $\left(Id + h_n q\right) - \exp\{h_n q\}$ as a function of 4 variables).} By combining \eqref{eq401}-\eqref{eq403}, we obtain that there exist 
$C_1 > 0, C_2 > 0$ such that for any continuous $f$ and
any $n \in \N^*$ we have
\begin{equation}
  \label{eq404}
  \frac{1}{h_n} \left\|\psi_n^h (f) - \overline \psi_n^h (f)\right\|_\infty
  \le C_1 \big(k(t_n) - k(t_{n+1})\big) + C_2 h_n \|f\|_\infty.
\end{equation}

By \eqref{eq404} and Proposition~\ref{shapley}(1,3) we have 
\begin{align*}
  \left\|
  v_{T_\infty,k} - 
  v^{\texttt{cont}}_{T_\infty,k} \right\|_\infty =&
  \left\|\prod_{i = 1}^\infty \psi^h_i (0) -
  \prod_{i = 1}^\infty \overline \psi^h_i (0) \right\|_\infty \\
  \le& \left\|\psi^h_1\left(\prod_{i = 2}^\infty \psi^h_i (0)\right) - 
  \overline \psi^h_1 \left(\prod_{i = 2}^\infty \psi^h_i (0)\right)
  \right\|_\infty + 
  \left\|\overline \psi^h_1 \left(\prod_{i = 2}^\infty \psi^h_i (0)\right) - 
  \overline \psi^h_1 \left(\prod_{i = 2}^\infty \overline \psi^h_i (0)\right)
  \right\|_\infty \\
  \le&\left\|\psi^h_1\left(\prod_{i = 2}^\infty \psi^h_i (0)\right) - 
  \overline \psi^h_1 \left(\prod_{i = 2}^\infty \psi^h_i (0)\right)
  \right\|_\infty + 
  \left\|\prod_{i = 2}^\infty \psi^h_i (0) - 
  \prod_{i = 2}^\infty \overline \psi^h_i (0) \right\|_\infty.
\end{align*}

By induction we obtain, for any $N \in \N^*$,
\begin{multline}
  \hspace{-0.65cm}
  \label{fpvv01}
  \left\|v_{T_\infty,k} - 
  v^{\texttt{cont}}_{T_\infty,k} \right\|_\infty \le 
  \sum_{m=1}^N \left(\left\|\psi^h_m
  \left(\prod_{i = m+1}^\infty \psi^h_i (0)\right) - 
  \overline \psi^h_m \left(\prod_{i = m+1}^\infty 
  \psi^h_i (0)\right)\right\|_\infty\right) + 
  \left\|\prod_{i = N+1}^\infty \psi^h_i (0) - 
  \prod_{i = N+1}^\infty \overline \psi^h_i (0) \right\|_\infty.
\end{multline}

Note that by Proposition~\ref{shapley}(3) we have, for any $m \in \N^*$ and any
initial time $t_m \in T_\infty$,
\begin{equation}
  \label{fpvv02}
  v_{T_\infty,k}(t_m,\cdot) = \prod_{i = m}^\infty \psi^h_i (0) \text{ and }
  v^{\texttt{cont}}_{T_\infty,k}(t_m,\cdot) = 
  \prod_{i = m}^\infty \overline \psi^h_i (0).
\end{equation}

Fix $\varepsilon > 0$. There is $S \in (0,T)$ such that, 
for any $t_m \ge S$ and any $\sup_{i\in\N^*} h_i$ small enough, we have
(we use a computation similar to the one in
Remark~\ref{important})
\begin{equation}
  \label{fpvv03}
  \|v_{T_\infty,k}(t_m,\cdot)\|_\infty \le
  \|g\|_\infty \left(\int_S^T k(t) dt + \varepsilon\right) \le
  2 \varepsilon \|g\|_\infty
  \;\; \text{ and } \;\;
  \|v^{\texttt{cont}}_{T_\infty,k}(t_m,\cdot)\|_\infty \le
  2 \varepsilon \|g\|_\infty.
\end{equation}

Also, if $\sup_{i\in\N^*} h_i$ is small enough, then by Remark~\ref{important}
we have for any $t_{m}$
\begin{equation}
  \label{fpvv04}
  \|v_{T_\infty,k}(t_m,\cdot)\|_\infty \le (1+\varepsilon) \|g\|_\infty.
\end{equation}
And analogously
$\|v^{\texttt{cont}}_{T_\infty,k}(t_m,\cdot)\|_\infty \le (1+\varepsilon) \|g\|_\infty.$

Now, let $N(T_\infty) \in \N^*$ be such that 
$t_{N(T_\infty) - 1} < S$ and $t_{N(T_\infty)} \ge S$.
Note that it depends on the partition $T_\infty$ of $[0,T).$
By \eqref{eq404}--\eqref{fpvv04}, there exist 
$C_1 > 0, C_2 > 0$ such that
\begin{align*}
  \left\|
  v_{T_\infty,k} - 
  v^{\texttt{cont}}_{T_\infty,k} \right\|_\infty 
  \le&
  \sum_{m=1}^{N(T_\infty)} \left(\left\|\psi^h_m
  \left(\prod_{i = m+1}^\infty \psi^h_i (0)\right) - 
  \overline \psi^h_m \left(\prod_{i = m+1}^\infty 
  \psi^h_i (0)\right)\right\|_\infty\right) 
  + 4 \varepsilon \|g\|_\infty
  \\ 
  \le& \sup_{i\in\N^*} h_i \cdot \sum_{m=1}^{N(T_\infty)}
  \big(C_1 (k(t_m) - k(t_{m+1})) + (1+\varepsilon) C_2 h_m \|g\|_\infty\big) 
  + 4 \varepsilon \|g\|_\infty \\
  \le&
  \sup_{i\in\N^*} h_i \cdot (C_1 k(0) + 
  (1+\varepsilon) C_2 t_{N(T_\infty)} \|g\|_\infty)
  + 4 \varepsilon \|g\|_\infty
  \xrightarrow{\sup_{i\in\N^*} h_i \to 0} 
  4 \varepsilon \|g\|_\infty 
  \xrightarrow{\varepsilon \to 0} 0.
\end{align*}

Now, the statement of the theorem follows directly from
\cite[Proposition~4.3]{Sor17}.
\end{proof}

\begin{remark}[The case of $T \neq +\infty$]
The above proof can be simplified if we assume that $T \neq +\infty$. 
In this case by 
\eqref{eq404}, \eqref{fpvv01}, \eqref{fpvv02}, \eqref{fpvv04} we have
\begin{align*}
  \left\|
  v_{T_\infty,k} - v^{\texttt{cont}}_{T_\infty,k} 
  \right\|_\infty 
  \le&
  \sum_{m=1}^{+\infty} \left(\left\|\psi^h_m
  \left(\prod_{i = m+1}^\infty \psi^h_i (0)\right) - 
  \overline \psi^h_m \left(\prod_{i = m+1}^\infty 
  \psi^h_i (0)\right)\right\|_\infty\right) \\
  \le& \sup_{i\in\N^*} h_i \cdot \sum_{m=1}^{+\infty}
  \big(C_1 (k(t_m) - k(t_{m+1})) + 
  (1+\varepsilon) C_2 h_m \|g\|_\infty\big) \\
  \le& \sup_{i\in\N^*} h_i \cdot (C_1 k(0) + 
  (1+\varepsilon) C_2 T \|g\|_\infty)
  \xrightarrow{\sup_{i\in\N^*} h_i \to 0} 0.
\end{align*}
\end{remark}





\section{State-blind stochastic games with stage duration}
\label{repeatedGames234}

In this section, we consider stochastic games in which players cannot observe the current state. In \S\ref{repeatedGames}, we recall the definition of state-blind stochastic games, and show that each state-blind stochastic game is equivalent to a stochastic game with perfect observation of the state. In \S\ref{def1}, we give a natural definition of games with stage duration in this framework. In \S\ref{jjjnt44}, we give a new result showing that in the case of state-blind stochastic games, the uniform limit $\lim_{\sup_{i\in\N^*} h_i \to 0} v_{T_\infty,k}$ exists and is the unique viscosity solution of some partial differential equation. The proof of this result is given in \S\ref{thh10}.

\subsection{The model of zero-sum state-blind stochastic games}
\label{repeatedGames}

A zero-sum state-blind stochastic game is played in the same way as a  
stochastic game with perfect observation of the state, but players 
cannot observe the current state. At the beginning of each stage, players
only remember the actions of players at the previous stage and 
the initial probability distribution on the states.

We still assume that the state space $\Omega$ and the action spaces $I, J$ are finite.

We can define the strategies, $\lambda$-discounted and finitely repeated 
games, total payoffs, values, in the same way as in \S\ref{stoch-intro}. 
In particular, a strategy of player 1 is a
collection of functions\\
$(i_1, j_1, i_2, j_2, \ldots, i_{n-1}, j_{n-1})
\mapsto \Delta(I),$
where $i_m \in I$ and $j_m \in J$.
The value is a function of a probability law $p_0$, according to which 
the initial state is chosen. It is known that the value exists if $\Omega, I, J$ are finite, see \cite{Sor02}.

Any state-blind stochastic game is equivalent to 
a stochastic game with perfect observation of the state.
Consider the following construction (cf. \cite[\S1.3]{Zil16}). 

Given a zero-sum state-blind stochastic game 
$G = (\Omega, I, J, \{g_m\}_{m\in\N^*}, \{P_m\}_{m \in \N^*})$, 
we define the stochastic game  
$\Gamma(G) = (\Delta(\Omega), I, J, \{g_m^\gamma\}_{m\in\N^*}, 
\{P_m^\gamma\}_{m \in\N^*})$
with perfect observation of the state. The function 
$g_m^\gamma : I \times J \times \Delta(\Omega) \to \R$ is defined by
$g_m^\gamma (i,j,p) = \sum\nolimits_{\omega \in \Omega} p(\omega) 
g_m(i, j, \omega)$.

Now let us define $\{P_m^\gamma\}_{m \in\N^*}$. If the current stage is $m$ and players have belief $p \in \Delta(\Omega)$ about the current state, then after playing $(i, j) \in I \times J$ their posterior belief that the current state is $\omega'$ is equal to $P_m (i,j)(p,\omega') := \sum\nolimits_{\omega \in \Omega} p(\omega) P_m(i, j)(\omega, \omega')$. The function
$P_m^\gamma : I \times J \to \Delta(\Omega) \times \Delta(\Omega)$
is defined by 
$$P_m^\gamma (i,j)(p,p') = 
\begin{cases}
1  &,\text{if } p'(\omega) = P_m (i, j) (p, \omega)
\text{ for all } \omega \in \Omega; \\
0  &,\text{otherwise}.
\end{cases}$$

Similarly, each strategy $s$ in $G$ has an analogous strategy 
$\Gamma(s)$ in $\Gamma(G)$.


\begin{definition}
A strategy $s$ in a state-blind stochastic game $G$ is said to be 
\emph{Markov} if the strategy
$\Gamma(s)$ in the stochastic game $\Gamma(G)$ with 
perfect observation of the state is Markov.
\end{definition}

\subsection{Zero-sum state-blind stochastic games with stage duration}
\label{def1}

As before, $T$ is either a positive number or $+\infty$, and $T_\infty$ is a partition of $[0,T)$; $h_n = t_{n+1}-t_n$ for each $n \in \N^*$.

\begin{definition}
Fix a zero-sum state-blind stochastic game 
$(\Omega, I, J, g, q)$, 
where $q$ is the kernel. 
The \emph{state-blind stochastic game with $n$-th stage duration $h_n$} 
is the state-blind stochastic game
$$(\Omega, I, J, \{h_m g\}_{m\in\N^*}, \{h_m q\}_{m\in\N^*}).$$
\end{definition}
We can define a total payoff for games with stage duration as in \S\ref{one-h}.

\begin{remark}[Why state-blind case is more difficult than the case of perfect observation of the state?]
If we denote by $T_h$ the Shapley operator of the game in which each stage has duration $h$, then in the case of perfect observation of the state we have
$T_h = h T_1 + (1-h)Id$.
This makes studying of such games relatively easy. In the case of state-blind stochastic games, such an equality no longer holds.
\end{remark}

\subsection{Theorem~\ref{thh1}: in the state-blind case, the uniform limit 
$\lim_{\sup_{i\in\N^*} h_i \to 0} v_{T_\infty,k}$ is the unique
viscosity solution of some differential equation}

\label{jjjnt44}

The following theorem is an analogue of Theorem~\ref{xmn444} for state-blind stochastic games.

\begin{theorem}
\label{thh1}
If $(\Omega, I, J, g, q)$ is a finite state-blind stochastic game, then the uniform limit \\
$\lim\limits_{\underset{h_1+h_2+\ldots = T}{\sup_{i\in\N^*} h_i \to 0}} 
v_{T_\infty,k}(t,p)$ exists and is the unique viscosity solution of the partial differential equation
\begin{equation}
\label{rlp3}
0 = \frac{d}{dt} v(t,p) + \texttt{Val}_{I\times J} [k(t) g(i,j,p) + 
\langle p * q(i,j), \nabla v(t,p) \rangle].
\end{equation}
\end{theorem}

For the discounted case the above equation is more simple:

\begin{corollary}
\label{thh3}
If $(\Omega, I, J, 
g, q)$ is a finite state-blind stochastic game, then the uniform limit \\
$\lim\limits_{\underset{h_1+h_2+\ldots = +\infty}{\sup_{i\in\N^*} h_i \to 0}} 
v_{T_\infty,\lambda}(p)$ exists and is the unique viscosity solution of the partial differential equation
$$\lambda v(p) = \texttt{Val}_{I\times J} [\lambda g(i,j,p) + \langle p * q(i,j), \nabla v(p) \rangle].$$
\end{corollary}

\begin{proof}
In \eqref{rlp3}, make a substitution $v(t,p) \mapsto e^{-\lambda t} v(p)$.
\end{proof}

\begin{remark}
	The above theorem is similar to \cite[Proposition~5.3]{Sor17}, which concerns another model of games with stage duration.
\end{remark}

\subsection{The proof of Theorem~\ref{thh1}}
\label{thh10}

Let us first state a technical lemma.

We denote for $x \in \Delta(I), y\in \Delta(J), \omega \in \Omega$
$$
  q(x,y) := \sum_{i \in I, j \in J} x(i) y(j) q(i,j) \qquad \text{ and } \qquad
  g(x,y,\omega) := \sum_{i \in I, j \in J} x(i) y(j) g(i,j,\omega).
$$

\begin{lemma}
  \label{lemma_pr}
The family $\{v_{T_\infty,k}(t,p)\}_{T_\infty}$ 
is equilipschitz-continuous and equibounded for all partitions
$T_\infty$ with $\sup h_i$ small enough, i.e.
there are positive constants
$C_1, C_2, C_3$ and there is $\delta \in (0, 1]$ such that for any 
$t^1, t^2 \in [0, T), p_1, p_2 \in \Delta (\Omega)$ 
and for any partition 
$T_\infty =\{t_n\}_{n\in\N^*}$ with $t_{n+1} - t_n \le \delta$, we have
\begin{align*}
  &|v_{T_\infty,k}(t^1,p_1) - 
  v_{T_\infty,k}(t^2,p_2)| \le
  C_1 \|p_1 - p_2\|_1 + C_2|t^1 - t^2|;\\
  &|v_{T_\infty,k}(t^1, p_1)| \le C_3.
\end{align*}
\end{lemma}

Lemma~\ref{lemma_pr} is an analogue of \cite[Proposition~3.11]{Sor17}, while the deduction of Theorem~\ref{thh1} from it is an analogue of \cite[Proposition~3.12]{Sor17}. The proof of Lemma~\ref{lemma_pr} is different from the proof of its analogue, while the deduction of Theorem~\ref{thh1} from it is almost identical to \cite[proof of Proposition~3.12]{Sor17}.

\begin{proof}[Proof of Lemma~\ref{lemma_pr}]
First, we prove equiboundedness. 
By Remark~\ref{important} (with $\varepsilon = 1$) 
we have for $\sup h_i$ small enough
$$|v_{T_\infty,k}(t^1, p_1)| \le 
|v_{T_\infty,k}(0, p_1)| \le
\|g\|_\infty \sum_{j=1}^\infty h_j k(t_j) 
\le 
2 \|g\|_\infty.$$
The rest of the proof is devoted to equilipschitz-continuity.

For $x \in \Delta(I), y\in \Delta(J)$, denote 
$P^h(x,y) := Id + h q(x,y).$
For all
$p \in \Delta(\Omega), \omega \in \Omega, x \in \Delta(I), y \in \Delta(J)$,
denote
$\widehat p^h(x,y)(\omega) := (p * P^h(x,y)) (\omega).$
If 
\begin{itemize}
\item For some $n \in \N^*$ we have $t_{n+1} - t_n = h$;
\item At the $n$-th stage players play a (mixed) action profile $(x,y)$;
\item The distribution of states at the start of the $n$-th stage is $p$,
\end{itemize}
then $\widehat p^h(x,y)$ is the distribution of states at the start of 
$(n+1)$-th stage. 

We have for all $h \in (0,1], x \in \Delta(I), y \in \Delta(J)$
$$\left\|\widehat p_1^h(x,y) - \widehat p_2^h(x,y) \right\|_1 \le 
\left\| (p_1 - p_2) * P^h(x,y) \right\|_1 = 
\left\| (P^h(x,y))^T \cdot (p_1 - p_2) \right\|_1,$$
where $p_1 - p_2$ is a vector column, and we have
$$\left\| (P^h(x,y))^T \cdot (p_1 - p_2) \right\|_1 \le
\|(P^h(x,y))^T\|_{\texttt{op}} \|p_1 - p_2\|_1,$$
where
$\|(P^h(x,y))^T\|_{\texttt{op}}$
is the operator norm of the operator $z \mapsto (P^h(x,y))^T \cdot z,$ i.e.
$$\left\|(P^h(x,y))^T\right\|_{\texttt{op}} = \sup_{z : \|z\|_1 = 1}
\left\|(P^h(x,y))^T \cdot z\right\|_1,$$
and since
$\left\|(P^h(x,y))^T \right\|_{\texttt{op}} = 1,$
we have
\begin{equation}
\left\|\widehat p_1^h(x,y) - \widehat p_2^h(x,y) \right\|_1 \le 
\left\| p_1 - p_2 \right\|_1.
\label{3788}
\end{equation}
Fix $t_n \in T_\infty$, and $p_1, p_2 \in \Delta(\Omega)$. 
By Proposition~\ref{shapley}(5) and Remark~\ref{rrff3}
there exists\footnote{Here, we consider the state-blind stochastic game as a stochastic game with perfect observation of the state, on the compact state space $\Delta(\Omega)$.} a profile of optimal Markov strategies $(\sigma_1, \tau_1)$ 
(respectively there exists a profile of optimal Markov strategies 
$(\sigma_2, \tau_2)$), 
if initial time is $t_n$ and initial distribution of states is $p_1$ 
(respectively $p_2$). For a strategy 
$(\sigma_i, \tau_i) \; (i = 1, 2)$, denote by
$(x_i^j,y_i^j)$ the profile of mixed actions played at the $j$-th stage 
(it depends on the $j$-th stage distribution of states $p_i^j$ and on the stage number $j$). 
We have
\begin{equation}
\left|v_{T_\infty,k}(t_n, p_1) - 
v_{T_\infty,k}(t_n, p_2) \right| \le  
\sum_{j = n}^\infty h_j k(t_j) 
\left|g(x_1^j,y_1^j,p_1^j) - g(x_2^j,y_2^j,p_2^j)\right|,
\label{975}
\end{equation}
where $p_i^{n} = p_i$, and $p^{j+1}_i = \widehat{p^j_i}^{h_j} (x^j_i,y^j_i)$ (for $i=1,2$ and $j \ge n$).
We have for any $j$ (assuming without loss of generality assume that
$g(x_1^j,y_1^j,p_1^j) \ge g(x_2^j,y_2^j,p_2^j)$)
\begin{align}
\hspace{-0.3cm}
\begin{split}
  \left|g(x_1^j,y_1^j,p_1^j) - g(x_2^j,y_2^j,p_2^j)\right| \le 
  &\left|g(x_1^j,y_2^j,p_1^j) - g(x_1^j,y_2^j,p_2^j)\right|
  = \left|\sum_{\omega \in \Omega} p_{1}^j(\omega ) g(x_1^j,y_2^j,\omega ) -
  \sum_{\omega  \in \Omega} p_{2}^j(\omega) g(x_1^j,y_2^j,\omega)\right| \\
  \le& \left\|g\right\|_\infty \sum_{\omega \in \Omega} 
  \left|p_{1}^j(\omega) - p_{2}^j(\omega)\right| 
  = \left\|g\right\|_\infty \left\|p_{1}^j - p_{2}^j\right\|_1
  \le \left\|g\right\|_\infty \left\|p_{1} - p_{2}\right\|_1,
\end{split}
\label{985}
\end{align}
where the last inequality follows from \eqref{3788}. 
By Remark~\ref{important} (with $\varepsilon = 1$) and by 
combining \eqref{975} and \eqref{985}, we obtain
\vspace{-0.4cm}
\begin{equation}
\left|v_{T_\infty,k}(t_n, p_1) - 
v_{T_\infty,k}(t_n, p_2) \right| \le  
\left\|g\right\|_\infty \left\|p_{1} - p_{2}\right\|_1 
\sum_{j = n}^\infty h_j k(t_j) \le 
2\left\|g\right\|_\infty \left\|p_{1} - p_{2}\right\|_1,
\label{rs1}
\end{equation}
where the second inequality holds if \red{$\sup h_i$} is small enough.

Fix $p \in \Delta(\Omega)$, and $t_n, t_m \in T_\infty$ with 
$t_m \ge t_n$. Let $(\sigma_1, \tau_1)$ (respectively $(\sigma_2, \tau_2)$) 
be a profile of optimal Markov strategies, if starting time is $t_n$ 
(respectively $t_m$) and initial distribution of states is $p$. 
For a strategy $(\sigma_i, \tau_i)$, denote by
$(x_i^j,y_i^j)$ the profile of mixed actions played at the $j$-th stage 
(it depends on the current distribution $p_i^j$ and on the stage number $j$). We have
\begin{align}
\begin{split}
\left|v_{T_\infty,k}(t_n, p) - 
v_{T_\infty,k}(t_m, p) \right| &\le  
\sum_{j = n}^{m-1} h_j k(t_j) 
\left|g(x_1^j,y_1^j,p_1^j)\right|
+ \sum_{j = m}^\infty h_j k(t_j) 
\left|g(x_1^j,y_1^j, p_1^j) - g(x_2^j,y_2^j,p_2^j)\right| \\
&\le \|k\|_\infty \|g\|_\infty |t_n - t_m|
+\sum_{j = m}^\infty h_j k(t_j)
\left|g(x_1^j,y_1^j, p_1^j) - g(x_2^j,y_2^j,p_2^j)\right|,
\end{split}
\label{9759}
\end{align}
where $p_1^n = p_2^m = p$ and 
$p^{j+1}_i = \widehat{p^j_i}^{h_j} (x^j_i,y^j_i)$ 
(for $i = 1, j \ge n$ or $i = 2, j \ge m$). 
There exists $\overline p \in \Delta (\Omega)$ such that
\begin{align}
\begin{split}
\|p_2^m - p_1^m\|_1 &= \left\|p - 
\left(\prod_{j=n}^{m-1}(1-h_j) p +
\left(1 - \prod_{j=n}^{m-1}(1-h_j)\right)\overline p\right)\right\|_1.
\end{split}
\label{97599}
\end{align}

By Lemma~\ref{lem1}(3) below we have
\begin{align}
  \begin{split}
  \left\|p - \left(\prod_{j=n}^{m-1}(1-h_j) p +
  \left(1 - \prod_{j=n}^{m-1}(1-h_j)\right)\overline p\right)\right\|_1 &\le
  \left\|\left(1 - \prod_{j=n}^{m-1}(1-h_j)\right)p \right\|_1
  + \left\|\left(1 - \prod_{j=n}^{m-1}(1-h_j)\right) \overline p \right\|_1 \\
  &\le \left\|\left(1 - \left(1-\sum_{j=n}^{m-1}h_j\right) \right)p \right\|_1 
  + \left\|\left(1 - \left(1-\sum_{j=n}^{m-1}h_j\right)\right) 
  \overline p \right\|_1 \\
  &\le \left|t_n -t_m \right| (\|p\|_1+\|\overline p\|_1) = 2 
  \left|t_n - t_m\right|.
  \end{split}
  \label{97599f4}
\end{align}
  
By combining \eqref{97599} and \eqref{97599f4}, we have
\begin{equation}
  \|p_2^m - p_1^m\|_1 \le 2 \left|t_n - t_m\right|.
  \label{97599g55}
\end{equation}

By Remark~\ref{important} (with $\varepsilon = 1$) and by 
combining \eqref{3788}, \eqref{9759}, and \eqref{97599g55}, we have
\begin{align}
\begin{split}
\left|v_{T_\infty,k}(t_n, p) - 
v_{T_\infty,k}(t_m, p) \right| &\le 
\|k\|_\infty \|g\|_\infty |t_n - t_m| + 
2 \left|t_n-t_m\right| 
\sum\nolimits_{j = n}^\infty h_j k(t_j) \\
&\le
\|k\|_\infty \|g\|_\infty |t_n - t_m| + 4 |t_n - t_m|\\
&=  
(\|k\|_\infty \|g\|_\infty+4) |t_n - t_m|.
\end{split}
\label{rs2}
\end{align}

By combining \eqref{rs1} and \eqref{rs2}, we have for any 
$p_1,p_2 \in \Delta(\Omega)$, $t^1 = t_n\in T_\infty$, and 
$t^2 = t_m\in T_\infty$
\begin{align}
\begin{split}
|v_{T_\infty,k}(t^1, p_1) - 
v_{T_\infty,k}(t^2, p_2) | &\le
|v_{T_\infty,k}(t^1, p_1) - 
v_{T_\infty,k}(t^1, p_2) | \\
&+
|v_{T_\infty,k}(t^1, p_2) - 
v_{T_\infty,k}(t^2, p_2)| \\
&\le
2 \left\|g\right\|_\infty \left\|p_{1} - p_{2}\right\|_1 + 
(\|k\|_\infty \|g\|_\infty+4) |t^1 - t^2|.
\end{split}
\label{a74999}
\end{align}

Now, we prove that this inequality holds for any $t^1,t^2 \in [0,T)$. Denote
$C_1 := 2 \left\|g\right\|_\infty$ and 
$C_2 := \|k\|_\infty \|g\|_\infty+4.$
Without loss of generality, assume that $t^1 \ge t^2$. 
By the definition of $v_{T_\infty,\lambda}$ there exist integers $n, m$ 
and numbers 
$\alpha \in [0,1], \beta \in [0,1]$ such that for any 
$p \in \Delta(\Omega)$ we have
\begin{align*}
v_{T_\infty,k}(t^1, p) &= 
\alpha v_{T_\infty,k}(t_n, p) + 
(1 - \alpha) v_{T_\infty,k}(t_{n+1}, p) \;\text{ and }\\
v_{T_\infty,k}(t^2, p) &= 
\beta v_{T_\infty,k}(t_m, p) + 
(1 - \beta) v_{T_\infty,k}(t_{m+1}, p).
\end{align*}
We have \\
\begin{align*}
|v_{T_\infty,k}(t^1, p_1) - 
v_{T_\infty,k}(t^1, p_2) | 
&\le \alpha |v_{T_\infty,k}(t_n, p_1) - 
v_{T_\infty,k}(t_n, p_2) | 
+(1-\alpha) |v_{T_\infty,k}(t_{n+1}, p_1) - 
v_{T_\infty,k}(t_{n+1}, p_2)| \\
&\le C_1\|p_1 - p_2\|_1.
\end{align*}
If $\beta \ge \alpha$, then 
$t_{n+1} > t_{n} \ge t_{m+1} > t_{m}$, and we have
\begin{align*}
&|v_{T_\infty,k}(t^1, p_2) - 
v_{T_\infty,k}(t^2, p_2)| 
\le \alpha |v_{T_\infty,k}(t_n, p_{2}) - 
v_{T_\infty,k}(t_m, p_2)| \\
+&(\beta-\alpha) |v_{T_\infty,k}(t_{n+1}, p_{2}) - 
v_{T_\infty,k}(t_{m}, p_2)|
+ (1-\beta) |v_{T_\infty,k}(t_{n+1}, p_{2}) - 
v_{T_\infty,k}(t_{m+1}, p_2)| \\
\le& \alpha(C_2|t_m-t_n|) +
(\beta-\alpha)(C_2|t_{m}-t_{n+1}|)
+ (1-\beta) (C_2|t_{m+1}-t_{n+1}|) \\
=& C_2|\alpha (t_{n} - t_{m}) + (\beta - \alpha)(t_{n+1} - t_{m})
+ (1 - \beta) (t_{n+1} - t_{m+1})| \\
=&C_2|\beta t_m + (1-\beta)t_{m+1} - 
(\alpha t_n +(1-\alpha)t_{n+1})| 
= C_2|t^2 - t^1|.
\end{align*}

If $\alpha \ge \beta$, then $t_{n+1} \ge t_{m+1}$ and $t_{n} \ge t_{m}$, and a computation similar to the above one shows that
$|v_{T_\infty,k}(t^1, p_2) - v_{T_\infty,k}(t^2, p_2)| \le C_2|t^2 - t^1|.$
Thus we proved \eqref{a74999} for any $t^1,t^2 \in [0,T)$, and hence we 
proved the lemma.
\end{proof}

Denote
$$U := \left\{\text{accumulation points of uniform (in both } 
p \text{ and } t 
\text{) limit }
\lim\limits_{\underset{h_1+h_2+\ldots = T}{\sup_{i\in\N^*} h_i \to 0}} 
v_{T_\infty,k}(t,p)\right\}.$$

\begin{proof}[Proof of Theorem~\ref{thh1}]
Fix a partition $T_\infty$. By Proposition~\ref{shapley}(4) we have
for any $t_n \in T_\infty$
\begin{equation}
\label{eq1}
v_{T_\infty,k}(t_n,p) = \texttt{Val}_{I\times J} 
\left[h_n k(t_n) g(i,j,p) + v_{T_\infty,k}(t_{n+1},
\widetilde p(i,j))\right],
\end{equation}
where $\widetilde p(i_n,j_n) = p*(Id + h_n q(i,j))$. 

By Lemma~\ref{lemma_pr} and the Arzelà–Ascoli theorem we have $U\neq \emptyset$.
We are going to prove that any $\mathcal U \in U$
is a viscosity solution of the partial differential equation
$$0 = \frac{d}{dt} V(t,p) + \texttt{Val}_{I \times J} 
[k(t) g(i,j,p) + \langle p * q(i,j), \nabla V(t,p) \rangle].$$

Let $\psi(t,p)$ be a $\mathcal C^1$ function such that $\mathcal{U} - \psi$ 
has a strict local maximum at 
$(\overline t, \overline p) \in [0,T) \times \Delta(\Omega)$. 
Consider a sequence of partitions $\{T_\infty(m)\}_{m \in \N^*}$ 
such that in the partition $T_\infty(m)$ we have 
$\sup_{i\in\N^*} h_i \to 0$ as $m \to \infty$, and such that a sequence 
$W_m = v_{T_\infty(m),\lambda}$ converging uniformly
to $\mathcal U$ as $m \to \infty$, and let 
$(t^*(m), p(m))$ be a maximizing (locally near $(\overline t, \overline p)$) 
sequence for $(W_m - \psi)(t,p)$, where $t^*(m) \in T_\infty(m)$. 
In particular, $(t^*(m), p(m))$ converges to $(\overline t, \overline p)$ as 
$m \to \infty.$ Given an optimal in \eqref{eq1} mixed strategy 
$x^*(m) \in \Delta(I)$, 
one has with 
$t^*(m)=t_n \in T_\infty(m)$
$$W_m(t_n, p(m)) \le E_{x^*(m),y}
\left[h_n k(t_n) g(i,j,p(m)) + W_m(t_{n+1},\widetilde p
(i,j))\right], \; \forall y \in \Delta(J).
$$

For $m$ large enough, the choice of $(t^*(m),p(m))$ implies
$$\psi(t_n,p(m)) - W_m(t_n, p(m)) \le 
\psi(t_{n+1},\widetilde p(i,j)) - W_m(t_{n+1},\widetilde p
(i,j)), \; \forall i \in I, j \in J.$$

By using the continuity of $k$ and $\psi$ being $\mathcal C^1$, 
and the Taylor's theorem for $\psi(t_{n+1},\cdot)$, one obtains 
for all $y \in \Delta(J)$
\begin{align*}
\psi(t_n,p(m)) &\le E_{x^*(m),y} \left[h_n k(t_n) g(i,j,p(m)) + 
\psi(t_{n+1},\widetilde p(i,j))\right] \\
&= E_{x^*(m),y} \left[h_n k(t_n) g(i,j,p(m)) + 
\psi(t_{n+1},p(m) * (Id + h_n q(i,j)))\right] \\
&\le h_n k(t_n) g(x^*(m),y,p(m)) + \psi(t_{n+1},p(m)) \\
&+ h_n E_{x^*(m),y} \langle p(m) * q(i,j), 
\nabla \psi(t_{n+1},p(m)) \rangle + o(h_n),
\end{align*}
where $o(h_n) / h_n \to 0$ when $h_n \to 0$.
This gives for all $y \in \Delta(J)$\\
\begin{multline*}
  0 \le h_n \frac{\psi(t_{n+1},p(m)) - \psi(t_{n},p(m))}{h_n} 
  + h_n k(t_n) g(x^*(m),y,p(m)) + \\
  h_n E_{x^*(m),y} \langle p(m) * q(i,j), 
  \nabla \psi(t_{n+1},p(m)) \rangle + o(h_n).
\end{multline*}

Hence by dividing by $h_n$ and taking the limit as $m \to \infty$, 
one obtains, for some accumulation point $x^* \in \Delta(I)$ 
(we use again the continuity of $k$ and 
$\psi$ being $\mathcal C^1$)
$$0 \le \frac{d}{dt} \psi(\overline t,\overline p) + 
k(\overline t) g(x^*,y,\overline p) + 
E_{x^*,y} \langle \overline p * q(i,j), 
\nabla \psi(\overline t,\overline p) \rangle + o(h_n) \quad \forall y \in 
\Delta(J).$$

Analogously one can prove that if $\psi(t,p)$ is a $\mathcal C^1$ function 
such that $\mathcal{U} - \psi$ has a strict local minimum at 
$(\overline t, \overline p) \in [0,T) \times \Delta(\Omega)$, and 
$y^* \in \Delta(J)$ is optimal in \eqref{eq1}, then
$$0 \ge \frac{d}{dt} \psi(\overline t,\overline p) + 
k(\overline t) g(x,y^*,\overline p) + 
E_{x,y^*} \langle \overline p * q(i,j), 
\nabla \psi(\overline t,\overline p) \rangle + o(h_n) \quad \forall x \in 
\Delta(I).$$

Thus $\mathcal U$ is a viscosity solution of 
$0 = \frac{d}{dt} V(t,p) + \texttt{Val}_{I \times J} 
[k(t) g(i,j,p) + \langle p * q(i,j), \nabla V(t,p) \rangle].$
The uniqueness follows from \cite[Proposition~3.9]{Sor17}.
\end{proof}

\begin{remark}
In the proof of Theorem~\ref{xmn444}, we used the Shapley equation to prove that if $\sup h_i$ is small, then the value $v_{T_\infty,k}$ of a stochastic game with stage duration is close to the value $v^{\texttt{cont}}_{T_\infty,k}$ of the discretization of a continuous-time Markov game, which is known to converge when $\sup h_i \to 0$. The analogous was not done in the proof of Theorem~\ref{thh1}, because for state-blind stochastic games the Shapley equation has a much more complicated structure. Namely, in this case instead of estimating the difference
$\langle(Id + h_n q(i,j))(\omega,\cdot)\,, f(\cdot)\rangle - 
\langle(\exp{\{h_n q(i,j)\}})(\omega,\cdot)\,, f(\cdot)\rangle,$
we need to estimate the difference
$f(t_{n+1},
\widetilde p_1(i,j)) - f(t_{n+1},
\widetilde p_2(i,j)),$
where 
$\widetilde p_1(i,j) = p*(Id + h_n q(i,j))$ and
$\widetilde p_2(i,j) = p* \exp{\{h_n q(i,j)\}}$.
It is not clear how to find an appropriate estimate (which is small in comparison with $\sup h_i$).
\end{remark}

\section{Discounted stochastic games with stage duration}
\label{disc}

When considering discounted stochastic games with stage duration, we always assume that $T=+\infty$, so that $T_\infty$ is a partition of $\R_+$.

In this section, we consider in more details the discounted case. In previous sections, we said that in a discounted game with $n$-th stage duration $h_n$, the total payoff is 
$\lambda \sum\nolimits_{i = 1}^\infty \exp\{-\lambda t_i\} h_i g_i$. However, previous articles \cite{Ney13} and \cite{SorVig16} about games with stage duration considered the total payoff 
$\lambda \sum\limits_{i = 1}^\infty (\prod\limits_{j = 1}^{i-1} (1-\lambda h_j) ) h_i g_i$.
We did not use this total payoff because it is not a particular case of Definition~\ref{ty3ee}. This section shows that if $\sup h_i \to 0$, then these two total payoffs are essentially equivalent.

For each $h\in(0,1]$, consider a discount factor $\alpha_h \in [0, 1)$.
We want to impose some natural condition on $\alpha_h$, which will allow us to
study the value when $\sup_{i} h_i$ tends to $0$. 
Such a condition is given by the following definition.

\begin{definition}(cf. \cite[p. 240]{Ney13}).
A family of $h$-dependent discount factors $\alpha_h$ is called 
\emph{admissible} if 
$\displaystyle{\lim\limits_{h \to 0+} \frac{\alpha_h}{h}}$ exists. 
The limit is called the \emph{asymptotic discount rate}.
\end{definition}

\begin{example}[Families of admissible discount factors] \
\begin{enumerate}
\item $\alpha_h = 1 - e^{- \lambda h}$ (Family with asymptotic discount 
rate $\lambda \in (0,+\infty)$);
\item $\alpha_h = \lambda h$ (Family with asymptotic discount 
rate $\lambda \in (0,1]$);
\item $\alpha_h = 
\begin{cases}
0, &\text{if } h > 1/\lambda; \\
\lambda h, &\text{if } h \le 1/\lambda.
\end{cases}$ (Family with asymptotic discount rate $\lambda \in (0,+\infty)$).
\end{enumerate}
\end{example}

Now, we define discounted games with stage duration.

\begin{definition}[$\lambda$-discounted stochastic game with stage duration]
	Let $\alpha_h$ be an admissible family of discount factors with 
	asymptotic discount rate $\lambda$. \\
  The \emph{$\lambda$-discounted stochastic game with $n$-th stage duration $h_n$} 
	is the stochastic game with payoff
	$\sum\limits_{i = 1}^\infty \left(\prod\limits_{j = 1}^{i-1} 
	(1-\alpha_{h_j}) \right) h_i g_i.$
	Denote by $V^\alpha_{T_\infty,\lambda}$ the value of such a game, and denote by\\
	$v^\alpha_{T_\infty,\lambda} = \left.V^\alpha_{T_\infty,\lambda} \middle/
	\sum\limits_{i = 1}^\infty \left(\prod\limits_{j = 1}^{i-1} 
	(1-\alpha_{h_j}) \right) h_i\right.$ 
	its normalization.
\end{definition}


\begin{remark}
	\label{stth3}
	In the above definition:
	\begin{enumerate}
		\item If $\alpha_h = 1 - e^{-\lambda h}$, then
		$\sum\limits_{i = 1}^\infty \left(\prod\limits_{j = 1}^{i-1} 
		(1-\alpha_{h_j}) \right) h_i g_i = 
		\sum\limits_{i = 1}^\infty \left(\prod\limits_{j = 1}^{i-1} 
		\exp\{-\lambda h_j\} \right) h_i g_i = \sum\limits_{i = 1}^\infty 
		\exp\{-\lambda t_i\} h_i g_i;$
		\item If $\alpha_h = \lambda h$, then
		$\sum\limits_{i = 1}^\infty \left(\prod\limits_{j = 1}^{i-1} 
		(1-\alpha_{h_j}) \right) h_i g_i = 
		\sum\limits_{i = 1}^\infty \left(\prod\limits_{j = 1}^{i-1} 
		(1-\lambda h_j) \right) h_i g_i$, and 
		$v_{T_\infty,\lambda} = \lambda V_{T_\infty,\lambda}$ 
		(it follows from the fact that
		$\sum\limits_{i = 1}^\infty \left(\prod\limits_{j = 1}^{i-1} 
		(1-\lambda h_j) \right) h_i = \frac{1}{\lambda}$,
		see Lemma~\ref{lem1}(2) for a proof).
	\end{enumerate}
\end{remark}

\begin{lemma}
\label{admissible2}
	Fix $\lambda \in (0, +\infty)$. Consider two admissible families $\alpha_h$ 
	and $\beta_h$ of discount factors with asymptotic discount rate $\lambda$. 
	Then for all families (parametrized by partitions $T_\infty$ of $\R_+$) of 
	streams $x (T_\infty) = (g_1, g_2, \ldots)$ with 
	$|g_i| \le C h_i \: (i \in \N^*)$, the difference 
	$\sum\limits_{i = 1}^\infty \left(\prod\limits_{j = 1}^{i-1} 
	(1-\alpha_j) \right) g_i - \sum\limits_{i = 1}^\infty 
	\left(\prod\limits_{j = 1}^{i-1} (1-\beta_j) \right) g_i$
	tends to $0$ as $\sup\nolimits_{i\in\N} h_i$ tends to $0$.
\end{lemma}

Before the proof of Lemma~\ref{admissible2}, we first state and prove 
the following simple lemma.

\begin{lemma}
  \label{lem1}
  Fix $\lambda \in (0, 1]$ and a sequence 
  $H_\infty = \{h_i\}_{i\in\N^*}$ with $h_i \in (0,1]$. We have:
  \begin{enumerate}
    \item $\sum\limits_{i = 1}^\infty h_i = +\infty 
    \;\Rightarrow\; \prod\limits_{i = 1}^\infty (1 - \lambda h_i) = 0;$
    \item If $\sum\limits_{i = 1}^\infty h_i = +\infty$, then 
    $\sum\limits_{i = 1}^\infty \left(\prod\limits_{j = 1}^{i-1} 
    (1 - \lambda h_j)\right) h_i = \frac{1}{\lambda}$;
    \item For all $n \in \N^*$ we have $\prod\limits_{i = 1}^n 
    (1 - \lambda h_i) \ge 1- \lambda \sum\limits_{i=1}^n h_i$.
  \end{enumerate}
\end{lemma}
  
\begin{proof}
  The first assertion is a standard result from elementary analysis. 
  We prove the second assertion. For each $i \in \N$ denote 
  $k_i = 1 - \lambda h_i$. We have
  $$\sum\limits_{i = 1}^\infty \left(\prod\limits_{j = 1}^{i-1} 
  (1 - \lambda h_j)\right) h_i = \left(\frac{1 - k_1}{\lambda} + 
  \frac{k_1 - k_1 k_2}{\lambda} + \frac{k_1 k_2 - k_1 k_ 2 k_3}{\lambda} + 
  \ldots\right) = \frac{1}{\lambda},$$
  where the last equality holds because 
  $\prod\nolimits_{i = 1}^\infty k_i = 0$ by the first assertion of the lemma. 
  
  We prove the third assertion by induction on $n \in \N^*$. 
  The case $n=1$ is immediate. 
  Assume that for $n = k$ the assertion holds, i.e. we have 
  $\prod\limits_{i = 1}^k (1 - \lambda h_i) \ge
  1- \lambda \sum\limits_{i=1}^k h_i$.
  For $n = k+1$, we have 
  \[\prod\limits_{i = 1}^{k+1} (1 - \lambda h_i) \ge 
  \left(1- \lambda \sum\limits_{i=1}^k h_i\right) (1 - \lambda h_{k+1}) \ge
  1- \lambda \sum\limits_{i=1}^{k+1} h_i. \qedhere\]
\end{proof}  

We are ready to prove Lemma~\ref{admissible2}.

\begin{proof}[Proof of Lemma~\ref{admissible2}]
Fix a family of streams $x (H_\infty) = (g_1, g_2, \ldots)$ with 
$|g_i| \le C h_i \; (i \in \N^*)$, and assume that $\alpha_h$ and $\beta_h$ 
are two families of discount factors with asymptotic discount rate $\lambda$. 
In that case we have $1 - \alpha_h = 1-\lambda h + m(h)$ and 
$1 - \beta_h = 1-\lambda h + n(h)$, where $m(h) / h \to 0$ and 
$n(h) / h \to 0$ as $h \to 0$. Consider the sum 
$$\sum\limits_{i = 1}^\infty \left(\prod\limits_{j = 1}^{i-1} 
(1 - \lambda h_j + k_i(h_j))\right) h_i,$$ where $k_i(h_j)$ is either 
$m(h_j)$ or $n(h_j)$. Now fix $h'$ such that for all $h$ with
$0 < h \le h'$ we have 
$\displaystyle{0 < \lambda \pm \frac{|m(h)|}{h} < 1}$ and 
$\displaystyle{0 < \lambda \pm \frac{|n(h)|}{h} < 1}$. For any sequence 
$H_\infty$ with $0 < h_i \le h'$ we have by the second assertion of 
Lemma~\ref{lem1}
\begin{align*}
\sum\limits_{i = 1}^\infty \left(\prod\limits_{j = 1}^{i-1} 
(1 - \lambda h_j + k_i(h_j))\right) h_i &\le
\sum\limits_{i = 1}^\infty \left(\prod\limits_{j = 1}^{i-1} 
\left[1-\left(\lambda -\sup\limits_{j\in\N} \left\{\frac{|m(h_j)|}{h_j}, 
\frac{|n(h_j)|}{h_j}\right\} \right) h_j\right]\right) h_i \\
&= \frac{1}{\lambda -\sup\limits_{j\in\N}
\left\{\frac{|m(h_j)|}{h_j},\frac{|n(h_j)|}{h_j}\right\}}; \\
\sum\limits_{i = 1}^\infty \left(\prod\limits_{j = 1}^{i-1} 
(1 - \lambda h_j + k_i(h_j))\right) h_i &\ge
\sum\limits_{i = 1}^\infty \left(\prod\limits_{j = 1}^{i-1} 
\left[1-\left(\lambda + \sup\limits_{j\in\N} \left\{\frac{|m(h_j)|}{h_j}, 
\frac{|n(h_j)|}{h_j}\right\} \right) h_j\right]\right) h_i \\
&= \frac{1}{\lambda + \sup\limits_{j\in\N}
\left\{\frac{|m(h_j)|}{h_j},\frac{|n(h_j)|}{h_j}\right\}}.
\end{align*}

Hence $\sum\limits_{i = 1}^\infty \left(\prod\limits_{j = 1}^{i-1} 
(1 - \lambda h_j + k_i(h_j))\right) h_i \to \frac{1}{\lambda}$ as 
$\sup\limits_{j\in\N} h_j \to 0$. Now we have
\begin{align*}
&\left|\sum\limits_{i = 1}^\infty \left(\prod\limits_{j = 1}^{i-1} 
(1 - \alpha_j) \right) g_i - \sum\limits_{i = 1}^\infty 
\left(\prod\limits_{j = 1}^{i-1} (1 - \beta_j) \right) g_i\right| \\
\le&
C \sum\limits_{i = 1}^\infty 
\left|\prod\limits_{j = 1}^{i-1} (1 - \lambda h_j + m(h_j)) h_i - 
\prod\limits_{j = 1}^{i-1}  (1 - \lambda h_j + n(h_j)) h_i\right| \\
=& C \sum\limits_{i = 1}^\infty
\left(\max\left\{\prod\limits_{j = 1}^{i-1} (1 - \lambda h_j + m(h_j)) h_i,
\prod\limits_{j = 1}^{i-1} (1 - \lambda h_j + n(h_j)) h_i\right\}\right)\\
-&C \sum\limits_{i = 1}^\infty
\left(\min\left\{\prod\limits_{j = 1}^{i-1} (1 - \lambda h_j + m(h_j)) h_i,
\prod\limits_{j = 1}^{i-1} (1 - \lambda h_j + n(h_j)) h_i\right\}\right)\\
=&C \sum\limits_{i = 1}^\infty
\left(\prod\limits_{j = 1}^{i-1} (1 - \lambda h_j + s_i(h_j))\right) h_i 
-C \sum\limits_{i = 1}^\infty 
\left(\prod\limits_{j = 1}^{i-1} (1 - \lambda h_j + r_i(h_j))\right) h_i,
\end{align*}
where $s_i(h_j)$ and $r_i(h_j)$ are either $m(h_j)$ or $n(h_j)$.
By the above discussion we have
\[\sum\limits_{i = 1}^\infty \prod\limits_{j = 1}^{i-1} 
(1 - \lambda h_j + s_i(h_j)) h_i - \sum\limits_{i = 1}^\infty 
\prod\limits_{j = 1}^{i-1}  (1 - \lambda h_j + r_i(h_j)) 
h_i \xrightarrow{\sup\nolimits h_j \to 0} 
\frac{1}{\lambda} - \frac{1}{\lambda} = 0. \qedhere\]
\end{proof}

\begin{corollary}
\label{mlvpp}
If $\alpha_h$ and $\beta_h$ are two families of discount factors with 
asymptotic discount rate $\lambda$, then we have:
\begin{align}
\lim_{\underset{h_1+h_2+\ldots = +\infty}{\sup_{i\in\N^*} h_i \to 0}} 
V^\alpha_{T_\infty,\lambda} &=
\lim_{\underset{h_1+h_2+\ldots = +\infty}{\sup_{i\in\N^*} h_i \to 0}} 
V^\beta_{T_\infty,\lambda}; \label{rx3001}\\ 
\lim_{\underset{h_1+h_2+\ldots = +\infty}{\sup_{i\in\N^*} h_i \to 0}} 
\lambda V^\alpha_{T_\infty,\lambda} =
\lim_{\underset{h_1+h_2+\ldots = +\infty}{\sup_{i\in\N^*} h_i \to 0}} 
v^\alpha_{T_\infty,\lambda} &= 
\lim_{\underset{h_1+h_2+\ldots = +\infty}{\sup_{i\in\N^*} h_i \to 0}} v
^{\beta}_{T_\infty,\lambda} =
\lim_{\underset{h_1+h_2+\ldots = +\infty}{\sup_{i\in\N^*} h_i \to 0}} 
\lambda V^\beta_{T_\infty,\lambda}. \label{rx3002}
\end{align}
\end{corollary}

\begin{proof}[Proof of Corollary~\ref{mlvpp}]
	\eqref{rx3001} follows from Lemma~\ref{admissible2}, and 
	\eqref{rx3002} follows from \eqref{rx3001} and Remark~\ref{stth3}.
\end{proof}

A particular case of the above lemma was established in \cite{Ney13} for games in which each stage has the same duration. 

\section{Final comments}
\label{the_final}

\begin{remark}[The case of infinite stochastic games with stage duration]
In this article, we assumed that action spaces are finite. However, it is not necessary, and we may assume that the action spaces are compact metric spaces, as long as the state space is still finite. Indeed, it is straightforward to define games with stage duration for this more general case. Under standard assumptions (see, for example, assumptions in \cite{Sor02}) on the payoff and transition probability functions, there is a value for the game with stage duration, for any fixed partition $T_\infty$.

Theorem~\ref{thh1} still holds in this more general setting, because 
the proof of Theorem~\ref{thh1} uses only the Shapley equation and the finiteness of the state space, and as long as the Shapley equation holds, the proof of Theorem~\ref{thh1} holds too.

Theorem~\ref{xmn444} also holds in this setting, because
it is based on a result from \cite{Sor17}, which still holds if action sets are compact metric spaces and the state space is finite.
\end{remark}

\begin{remark}[Stochastic games with public signals and its limit values]
\label{great_apple}
In this article, we considered two types of games: stochastic 
games with perfect observation of the state and state-blind stochastic games.
We may also consider an intermediate case of stochastic games with public 
signals. In such games players are given a public signal that depends on 
the current state, but they may not observe the state itself.
It is possible to give a natural definition of games 
with stage duration and public signals, which is done in \cite{Nov24}.
A distinctive feature of games with stage duration and public signals is the fact that there is no connection between the limit value (as the discount factor $\lambda$ tends to $0$) of a game with stage duration $1$ and limit values of corresponding games with vanishing stage duration. In the case of perfect observation of the state, there is a connection, see \cite[\S7.3]{SorVig16}.
\end{remark}

Let us finish with two propositions that show the connection between the model of games with stage duration presented here and the model of games with stage duration from \cite{Sor17}. Recall that $v^{\texttt{cont}}_{T_\infty,k}$ is the value of the discretization of a continuous-time Markov game, see \S\ref{cc01}.

\begin{proposition}
\label{h3er}
Fix the discretization of a continuous-time finite Markov game $(\Omega, I, J, g, q_1)$, where $q_1$ is an infinitesimal generator, and fix a stochastic game $(\Omega, I, J, g, q_2)$, where $q_2$ is the kernel. If $q_1 = q_2 = q$, then the uniform limits
$\lim\limits_{\underset{h_1+h_2+\ldots = T}{\sup_{i\in\N^*} h_i \to 0}} 
v^{\texttt{cont}}_{T_\infty,k}$ and 
$\lim\limits_{\underset{h_1+h_2+\ldots = T}{\sup_{i\in\N^*} h_i \to 0}} 
v_{T_\infty,k}$ both exist and are equal.
\end{proposition}

\begin{proof}
By  Theorem~\ref{xmn444} and \cite[Proposition~4.3]{Sor17} any of these limit functions is the unique viscosity solution of
$0 = \frac{d}{dt} v(t,\omega) + \texttt{Val}_{I\times J} [k(t) g(i,j,\omega) + 
\langle q(i,j)(\omega,\cdot)\,, v(t,\cdot)\rangle].$
\end{proof}

\begin{proposition}
  \label{4cd6YYu}
  Fix the discretization of a continuous-time state-blind finite Markov game $(\Omega, I, J, g, q_1)$, where $q_1$ is its infinitesimal generator, and fix a state-blind stochastic game $(\Omega, I, J, g, q_2)$, where $q_2$ is its kernel. If $q_1 = q_2 = q$, then the uniform limits 
  $\lim\limits_{\underset{h_1+h_2+\ldots = +\infty}{\sup_{i\in\N^*} h_i \to 0}} 
  v^{\texttt{cont}}_{T_\infty,k}$ and 
  $\lim\limits_{\underset{h_1+h_2+\ldots = +\infty}{\sup_{i\in\N^*} h_i \to 0}} 
  v_{T_\infty,k}$ both exist and are equal.
\end{proposition}

\begin{proof}
By  Theorem~\ref{xmn444} and \cite[Proposition~5.3]{Sor17} any of these limit functions is the unique viscosity solution of
$0 = \frac{d}{dt} v(t,p) + \texttt{Val}_{I\times J} [k(t) g(i,j,p) + 
\langle p * q(i,j), \nabla v(t,p) \rangle]$.
\end{proof}

\section{Acknowledgements}
The author is grateful to his research advisor Guillaume Vigeral for constant
attention to this work. 
The author is grateful to Sylvain Sorin for useful discussions.
The author is grateful to anonymous reviewers for many helpful comments.

\bibliographystyle{alphaurl}
\bibliography{ref}

\begin{thebibliography}{GHL05}

\bibitem[GHL03]{GuoHer03}
Xianping Guo and On{\'e}simo Hern{\'a}ndez-Lerma.
\newblock Zero-sum games for continuous-time markov chains with unbounded
  transition and average payoff rates.
\newblock {\em Journal of Applied Probability}, 40(2):327--345, 2003.
\newblock \href {https://doi.org/10.1239/jap/1053003547}
  {\path{doi:10.1239/jap/1053003547}}.

\bibitem[GHL05]{GuoHer05}
Xianping Guo and On{\'e}simo Hern{\'a}ndez-Lerma.
\newblock Zero-sum continuous-time markov games with unbounded transition and
  discounted payoff rates.
\newblock {\em Bernoulli}, 11(6):1009--1029, 2005.
\newblock \href {https://doi.org/10.3150/bj/1137421638}
  {\path{doi:10.3150/bj/1137421638}}.

\bibitem[LRS19]{LarRenSor19}
Rida Laraki, J\'er\^{o}me Renault, and Sylvain Sorin.
\newblock {\em Mathematical Foundations of Game Theory}.
\newblock Universitext. Springer International Publishing, 2019.
\newblock \href {https://doi.org/10.1007/978-3-030-26646-2}
  {\path{doi:10.1007/978-3-030-26646-2}}.

\bibitem[MSZ15]{MerSorZam15}
Jean-Fran{\c c}ois Mertens, Sylvain Sorin, and Shmuel Zamir.
\newblock {\em Repeated Games}.
\newblock Econometric Society Monographs. Cambridge University Press, 2015.
\newblock \href {https://doi.org/10.1017/CBO9781139343275}
  {\path{doi:10.1017/CBO9781139343275}}.

\bibitem[Ney13]{Ney13}
Abraham Neyman.
\newblock Stochastic games with short-stage duration.
\newblock {\em Dynamic Games and Applications}, 3(2):236--278, 2013.
\newblock \href {https://doi.org/10.1007/s13235-013-0083-x}
  {\path{doi:10.1007/s13235-013-0083-x}}.

\bibitem[Nov24]{Nov24}
Ivan Novikov.
\newblock Limit value in zero-sum stochastic games with vanishing stage
  duration and public signals.
\newblock Preprint, 2024.
\newblock \href {https://doi.org/10.48550/arXiv.2401.10572}
  {\path{doi:10.48550/arXiv.2401.10572}}.

\bibitem[Sha53]{Sha53}
Lloyd~Stowell Shapley.
\newblock Stochastic games.
\newblock {\em Proceedings of the National Academy of Sciences},
  39(10):1095--1100, 1953.
\newblock \href {https://doi.org/10.1073/pnas.39.10.1095}
  {\path{doi:10.1073/pnas.39.10.1095}}.

\bibitem[Sor02]{Sor02}
Sylvain Sorin.
\newblock {\em A First Course on Zero-Sum Repeated Games}, volume~37 of {\em
  Math\'ematiques et Applications}.
\newblock Springer-Verlag Berlin Heidelberg, first edition, 2002.
\newblock \href {https://doi.org/10.1007/978-3-030-26646-2}
  {\path{doi:10.1007/978-3-030-26646-2}}.

\bibitem[Sor18]{Sor17}
Sylvain Sorin.
\newblock Limit value of dynamic zero-sum games with vanishing stage duration.
\newblock {\em Mathematics of Operations Research}, 43(1):51--63, 2018.
\newblock \href {https://doi.org/10.1287/moor.2017.0851}
  {\path{doi:10.1287/moor.2017.0851}}.

\bibitem[SV16]{SorVig16}
Sylvain Sorin and Guillaume Vigeral.
\newblock Operator approach to values of stochastic games with varying stage
  duration.
\newblock {\em International Journal of Game Theory}, 45(1):389--410, 2016.
\newblock \href {https://doi.org/10.1007/s00182-015-0512-8}
  {\path{doi:10.1007/s00182-015-0512-8}}.

\bibitem[Zac64]{Zac64}
Lars~Erik Zachrisson.
\newblock {\em Advances in Game Theory. (AM-52), Volume 52}, pages 211--254.
\newblock Princeton University Press, 1964.
\newblock \href {https://doi.org/10.1515/9781400882014-014}
  {\path{doi:10.1515/9781400882014-014}}.

\bibitem[Zil16]{Zil16}
Bruno Ziliotto.
\newblock Zero-sum repeated games: Counterexamples to the existence of the
  asymptotic value and the conjecture
  $\operatorname{maxmin}=\operatorname{lim}v_{n}$.
\newblock {\em The Annals of Probability}, 44(2):1107 -- 1133, 2016.
\newblock \href {https://doi.org/10.1214/14-AOP997}
  {\path{doi:10.1214/14-AOP997}}.

\end{thebibliography}

\end{document}